\documentclass[12pt,a4paper,reqno]{amsart}

\usepackage[T1]{fontenc}
\usepackage[utf8]{inputenc}
\usepackage{lmodern}
\usepackage{microtype}
\usepackage{mathtools,amssymb}
\usepackage{tikz}
\usepackage[hidelinks]{hyperref}

\allowdisplaybreaks[2]
\numberwithin{equation}{section}

\newtheorem{theorem}{Theorem}[section]
\newtheorem{proposition}[theorem]{Proposition}
\newtheorem{lemma}[theorem]{Lemma}
\newtheorem{corollary}[theorem]{Corollary}
\theoremstyle{definition}
\newtheorem{definition}[theorem]{Definition}

\theoremstyle{remark}
\newtheorem{remark}[theorem]{Remark}

\newcommand{\R}{\mathbb R}
\newcommand{\Lah}{\mathcal L}
\newcommand{\Match}{\mathcal M}
\newcommand{\e}{\mathrm e}

\title[Real-rootedness of $\tau$ under graph joins]
{Real-rootedness of the $\tau$-polynomial under graph joins}

\author{Mingyang Kang}
\address[Mingyang Kang]{Center for Combinatorics, LPMC, 
Nankai University, Tianjin 300071, P.R. China}
\email{2120240005@mail.nankai.edu.cn}

\author{Zhixin Liu}
\address[Zhixin Liu]{School of Mathematics, Tianjin University,
Tianjin 300350, P.R. China}
\email{zhixinliu@tju.edu.cn}

\author{Sophie C.C. Sun}
\address[Sophie C.C. Sun]{Department of Mathematics,
Tianjin University of Finance and Economics,
Tianjin 300222, P.R. China}
\email{sophiesun@tjufe.edu.cn}

\author{Philip B. Zhang}
\address[Philip B. Zhang]{College of Mathematical Sciences \& Institute of Mathematics and Interdisciplinary Sciences,
Tianjin Normal University,
Tianjin 300387, P.R. China}
\email{zhang@tjnu.edu.cn}


\begin{document}

\begin{abstract}
For a simple graph $G$ with $n$ vertices, write its chromatic polynomial in
 the rising factorial basis as
\[
 \chi_G(x)=\sum_{i=0}^{n}(-1)^{n-i}c_i(G)\langle x\rangle_i,
 \
\]
where $ \langle x\rangle_i=x(x+1)\cdots(x+i-1).$ The associated $\tau$-polynomial 
\[
 \tau_G(x)=\sum_{i=0}^{n}c_i(G)x^i
\]
  was defined and  systematically investigated by Brenti in 1992. In this paper, we
prove that if the $\tau$-polynomials of two vertex-disjoint simple graphs
$G$ and $H$ have only real zeros, then the $\tau$-polynomial of their join
$G\vee H$ has only real zeros. This settles  a conjecture posed by Brenti, Royle 
and Wagner since  1994.
\end{abstract}

\maketitle

\section{Introduction}\label{sec:introduction}

Throughout this paper, a graph $G=(V,E)$ is finite and simple.  It is well known that the
chromatic polynomial $\chi_G(x)$ of a graph with $n$ vertices is monic of
degree $n$.  For $i\geq0$, write
\[
 \langle x\rangle_i=x(x+1)\cdots(x+i-1)\ \text{for $i\geq 1$},
 \qquad \langle x\rangle_0=1,
\]
for the $i$-th rising factorial.  There are unique coefficients $c_i(G)$
such that
\begin{equation*}
 \chi_G(x)=\sum_{i=0}^{n}(-1)^{n-i}c_i(G)\langle x\rangle_i.
\end{equation*}
The polynomial
\[
\tau_G(x)=\sum_{i=0}^{n}c_i(G)x^i
\]
was introduced by Brenti \cite[Section~5]{Brenti92}, where fundamental  properties of $\tau_G(x)$ have been developed.

The \emph{join} $G\vee H$ of two vertex-disjoint graphs is obtained from
their disjoint union by adding every edge between $V(G)$ and $V(H)$.  

The goal of this paper is to 
prove the following conjecture by Brenti, Royle and Wagner {\cite{BRW94}} proposed in 1994.

\begin{theorem}[{\cite[Conjecture~6.6]{BRW94}}]\label{thm:graph-join}
Let $G$ and $H$ be vertex-disjoint graphs.  If both $\tau_G(x)$ and
$\tau_H(x)$ have only real zeros, then $\tau_{G\vee H}(x)$ has only real
zeros.
\end{theorem}

To prove Theorem \ref{thm:graph-join}, we shall define a star product $\star$ on
$\R[x]$ using the Lah numbers. A crucial observation is 
\begin{equation}\label{cb-8-nj}
 \tau_{G\vee H}(x)=\tau_G(x)\star\tau_H(x),   
\end{equation}
see  Proposition \ref{prop:join-star}, which is the starting point of this paper. 

The following is our main result, which, combined with \eqref{cb-8-nj}, will allow us to complete the proof of Theorem  \ref{thm:graph-join}. 

\begin{theorem}\label{thm:abstract-main}
Let $f,g\in\R[x]\setminus\{0\}$.  Suppose that $x$ divides both $f$ and
$g$, and that every zero of each polynomial belongs to $(-\infty,0]$.
Then every zero of $f\star g$ also belongs to $(-\infty,0]$.
\end{theorem}

Indeed, by  \cite[Corollary~5.4]{Brenti92}, if $G$ is nonempty, then
\begin{equation*}
 c_0(G)=0,\qquad c_j(G)\in\mathbb Z_{>0}\ (1\leq j\leq n),
 \qquad c_n(G)=1.
\end{equation*}
Consequently, $x$ divides $\tau_G(x)$, and every real zero of $\tau_G(x)$
is nonpositive.  So, combining Theorem~\ref{thm:abstract-main} with 
\eqref{cb-8-nj} leads to a proof of  Theorem \ref{thm:graph-join}.

Our proof of Theorem \ref{thm:abstract-main} mainly relies on two ingredients. 
The first one is a combinatorial formula for the product $x^p \star x^q$, see Theorem~\ref{thm:monomial-model}. More specifically, the formula says that $x^p \star x^q$ can be realized as a weighted counting of matchings over a certain bipartite graph $\Gamma_{p,q}$. As an application, we may obtain a combinatorial formula for $f\star g $  in terms of matchings over a weighted  bipartite graph determined by $\Gamma_{p,q}$ as well as the   real roots of $f$ and $g$, see Theorem \ref{thm:weighted-model}.   The second ingredient is a real-rootedness property concerning a signed counting of matchings over a weighted graph  due to  Heilmann and Lieb \cite{HL72}, see Theorem 
\ref{thm:heilmann-lieb}. This, combined with Theorem \ref{thm:weighted-model}, will  yield a proof of Theorem \ref{thm:abstract-main}. 

This paper is organized as follows. Section \ref{sec:lah} defines the star product on polynomials based on the Lah numbers. In particular, we justify the relation \eqref{cb-8-nj} in Proposition \ref{prop:join-star}. In Section \ref{sec:matching}, we prove the combinatorial formula for the product $x^p \star x^q$ which is achieved by   counting   matchings over a bipartite graph $\Gamma_{p,q}$. Section \ref{se-4-k} is devoted to a proof of   Theorem \ref{thm:abstract-main}.

\section{The Lah transform and a star product}\label{sec:lah}

For $i\geq0$, let
\[
 (x)_i=x(x-1)\cdots(x-i+1)\ \text{for $i\geq 1$},
 \qquad (x)_0=1,
\]
denote the $i$-th falling factorial.  The unsigned Lah numbers are
\[
 L(n,k)=\frac{n!}{k!}\binom{n-1}{k-1}
 \qquad(1\leq k\leq n),
\]
with $L(0,0)=1$ and $L(n,k)=0$ otherwise.  They are the connection
coefficients between the rising and falling factorial bases
\cite[p.~156, item~2]{Comtet}:
\begin{equation*}
 \langle x\rangle_n=\sum_{k=0}^{n}L(n,k)(x)_k.
\end{equation*}

Let $\Lah:\R[y]\to\R[x]$ be the linear transformation
\begin{equation*}
 \Lah(y^n)=\sum_{k=0}^{n}L(n,k)x^k.
\end{equation*}
Its matrix in the monomial bases is unitriangular, so $\Lah$ is an
isomorphism.  The inverse matrix gives
\begin{equation}\label{eq:Lah-inverse-intro}
 \Lah^{-1}(x^n)=\sum_{k=0}^{n}(-1)^{n-k}L(n,k)y^k,
\end{equation}
and, equivalently,
\begin{equation}\label{eq:inverse-factorial-change}
 (x)_n=\sum_{k=0}^{n}(-1)^{n-k}L(n,k)\langle x\rangle_k.
\end{equation}
For $f,g\in\R[x]$, define
\begin{equation}\label{eq:star-intro}
 f\star g=\Lah\bigl(\Lah^{-1}(f)\Lah^{-1}(g)\bigr).
\end{equation}
By definition, it is easily seen that  $\star$ is bilinear, commutative, associative, and has identity $1$.

Let us   illustrate the star product 
by considering $x^2\star x^3 $. 
By \eqref{eq:Lah-inverse-intro},
we compute that 
\[
 \Lah^{-1}(x^2)=y^2-2y,
 \qquad
 \Lah^{-1}(x^3)=y^3-6y^2+6y.
\]
So we have 
\begin{align*}
 x^2\star x^3
 &=\Lah\bigl((y^2-2y)(y^3-6y^2+6y)\bigr)\\
 &=x^5+12x^4+42x^3+48x^2+12x.
\end{align*}

\begin{remark}
The operator $\Lah$ and the product in \eqref{eq:star-intro} also appeared in~\cite{Jum25}. It should also be noted that this product is different from the star product used in
\cite[Proposition~5.2]{BRW94}.
\end{remark}

The $\sigma$-polynomial of $G$ is defined by 
\[
\sigma_G(x)=\sum_{i=0}^{n}a_i(G)x^i,
\]
where
\begin{equation}\label{nxj-08}
  \chi_G(x)=\sum_{i=0}^{n}a_i(G)(x)_i.   
\end{equation}
We refer to  \cite{Brenti92,BRW94} for basic properties concerning the $\sigma$-polynomials.

\begin{proposition}\label{prop:join-star}
We have the following statements.
\begin{itemize}
    \item[(1)] For every graph $G$ with $n$ vertices, one has 
\begin{equation}\label{bn-09-8}
\tau_G(x)=\Lah((-1)^n\sigma_G(-y)).    
\end{equation}

    \item[(2)] 
     If $G$ and $H$ are
vertex-disjoint, then
\begin{equation*}
 \tau_{G\vee H}(x)=\tau_G(x)\star\tau_H(x).
\end{equation*}
\end{itemize}

\end{proposition}

\begin{proof}
(1) Substituting \eqref{eq:inverse-factorial-change} into \eqref{nxj-08} gives
\begin{align*}
 \chi_G(x)
 &=\sum_{i=0}^{n}a_i(G)
   \left(\sum_{k=0}^{i}(-1)^{i-k}L(i,k)\langle x\rangle_k\right)\\
 &=\sum_{k=0}^{n}(-1)^{n-k}
   \left(\sum_{i=k}^{n}(-1)^{n-i}a_i(G)L(i,k)\right)
   \langle x\rangle_k.
\end{align*}
This implies that the coefficient $c_k(G)$ of $x^k$
  in $\tau_G(x)$ is 
\begin{equation}\label{gn-n-0}
    c_k(G)=\sum_{i=k}^{n}(-1)^{n-i}a_i(G)L(i,k).
\end{equation}
On the other hand,
\begin{align*}
 \Lah((-1)^n\sigma_G(-y))&=\sum_{i=0}^n(-1)^{n-i}a_i(G)\Lah(y^i)\\
 &=\sum_{i=0}^n(-1)^{n-i}a_i(G)\left(\sum_{k=0}^i L(i,k)x^k\right)\\
 &=\sum_{k=0}^n \left(\sum_{i=k}^{n}(-1)^{n-i}a_i(G)L(i,k)\right)x^k,
\end{align*}
which, together with \eqref{gn-n-0}, gives \eqref{bn-09-8}. 
  
(2) Suppose that $H$ has $m$ vertices.  By \cite[Theorem 3.13]{Brenti92},
\[
\sigma_{G\vee H}(x)=\sigma_G(x) \sigma_H(x).
\]
Hence, using \eqref{bn-09-8} and the definition of the star product, we obtain
\begin{align*}
 \tau_{G\vee H}(x)&=\Lah\left((-1)^{n+m}\sigma_{G\vee H}(-y)\right)\\
 &=\Lah\left((-1)^{n}\sigma_{G}(-y)\cdot (-1)^{m}\sigma_{H}(-y)\right)\\
 &=\Lah\left(\Lah^{-1}(\tau_G(x))\cdot \Lah^{-1}(\tau_H(x))\right)=\tau_G(x)\star\tau_H(x),
\end{align*}
as desired. 
\end{proof}

\section{A combinatorial formula for \texorpdfstring{$x^p\star x^q$}{monomial star products}}\label{sec:matching}

The aim of this section is to give a   combinatorial formula for the product  $x^p\star x^q$ for $p,q\geq 1$, which is based on a weighted  counting of matchings of a certain bipartite graph. 

\subsection{Statement of the main result}

Fix $p,q\geq1$.  Let
\[
 \mathsf R_p=\{r_1,\ldots,r_p\},\qquad
 \mathsf B_q=\{b_1,\ldots,b_q\},
 \qquad r_\ast=r_p,\quad b_\ast=b_q.
\]
We also write  $\mathsf R_0=\mathsf R_p\setminus\{r_\ast\}$ and
$\mathsf B_0=\mathsf B_q\setminus\{b_\ast\}$.  Let
\[
\widehat{\mathsf R}_0=\{\widehat{r}_1,\ldots, \widehat{r}_{p-1}\}\ \ \text{and}\ \ \widehat{\mathsf B}_0=\{\widehat{b}_1,\ldots, \widehat{b}_{q-1}\}
\]
be the disjoint copies of $\mathsf R_0$ and $\mathsf B_0$.

\begin{definition}\label{def:Gamma}
Let $\Gamma_{p,q}$ be the bipartite graph with vertex parts
\[
 \begin{aligned}
 X_{p,q}=\mathsf R_p\sqcup\mathsf B_q,\ \ \ 
 Y_{p,q}=\widehat{\mathsf B}_0\sqcup
          \widehat{\mathsf R}_0\sqcup\{\omega\},
 \end{aligned}
\]
and edges
\begin{align*}
 E(\Gamma_{p,q})={}&
 \bigl\{\{r,\widehat b\}:r\in\mathsf R_p,
                   b\in\mathsf B_0\bigr\}\\
 &\sqcup
 \bigl\{\{b,\widehat r\}:b\in\mathsf B_q,
                   r\in\mathsf R_0\bigr\}\\
 &\sqcup
 \bigl\{\{v,\omega\}:v\in X_{p,q}\bigr\}.
\end{align*}
In other words, $\Gamma_{p,q}$ is the disjoint union of the complete bipartite graphs $K_{p,q-1}$ and 
the $K_{q,p-1}$, together with a vertex $\omega$   adjacent to every vertex in 
$X_{p,q}$.  See Figure~\ref{fig:Gamma23} for an illustration of $\Gamma_{p,q}$ with $(p,q)=(2,3)$.
\end{definition}

\begin{figure}[htbp]
\centering
\begin{tikzpicture}[
    dot/.style={circle,fill=black,inner sep=1.7pt},
    edge/.style={line width=0.55pt}
]
% The K_{2,2} component.
\node[dot,label=left:$r_1$] (r1) at (0,3.3) {};
\node[dot,label=left:$r_2$] (r2) at (0,2.4) {};
\node[dot,label=right:$\widehat b_1$] (hb1) at (4.2,3.3) {};
\node[dot,label=right:$\widehat b_2$] (hb2) at (4.2,2.4) {};
\foreach \r in {r1,r2}
  \foreach \b in {hb1,hb2}
    \draw[edge] (\r)--(\b);

% The K_{3,1} component.
\node[dot,label=left:$b_1$] (b1) at (0,1.1) {};
\node[dot,label=left:$b_2$] (b2) at (0,0.25) {};
\node[dot,label=left:$b_3$] (b3) at (0,-0.60) {};
\node[dot,label=right:$\widehat r_1$] (hr1) at (4.2,0.25) {};
\foreach \b in {b1,b2,b3}
  \draw[edge] (\b)--(hr1);

% The universal vertex.
\node[dot,label=right:$\omega$] (w) at (5.05,-1.65) {};
\foreach \v in {r1,r2,b1,b2,b3}
  \draw[edge] (\v)--(w);
\end{tikzpicture}
\caption{The bipartite graph $\Gamma_{2,3}$.}
\label{fig:Gamma23}
\end{figure}
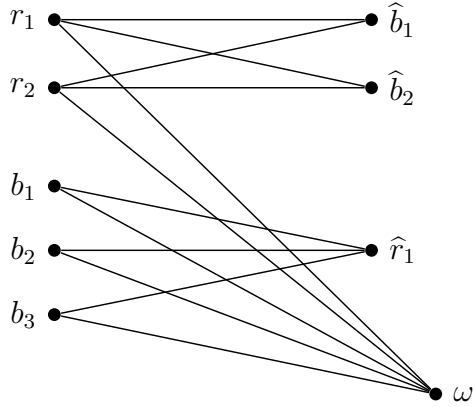

For a   graph $B$, let $\Match(B)$ denote the set of matchings on $B$.  Define
\[
 \Psi_{p,q}(x)
   =\sum_{M\in\Match(\Gamma_{p,q})}x^{p+q-|M|},
\]
where $|M|$ is the number of edges in $M$. Our main result in this section can be stated as follows.

\begin{theorem}\label{thm:monomial-model}
For all $p,q\geq1$,
\begin{equation}\label{eq:monomial-model}
 x^p\star x^q=\Psi_{p,q}(x).
\end{equation}
\end{theorem}

\subsection{Several lemmas for the proof of Theorem \ref{thm:monomial-model}}

We divide the proof of Theorem \ref{thm:monomial-model} into several  lemmas.   
To state the first lemma, let $\Gamma_{p,q}-\omega$ be the bipartite graph obtained from $\Gamma_{p,q}$ by removing the vertex $\omega$ along with all edges adjacent to $\omega$. Define
\[
 \Psi^{\circ}_{p,q}(x)
=\sum_{M\in\Match(\Gamma_{p,q}-\omega)}x^{p+q-|M|}.
\]

The first lemma gives  a relation connecting $ \Psi_{p,q}(x)$ and $\Psi^{\circ}_{p,q}(x)$. 

\begin{lemma}\label{lem:universal}
For   $p,q\geq1$,
\begin{equation}\label{eq:universal}
 \Psi_{p,q}(x)=(1+\partial_x)\Psi^{\circ}_{p,q}(x).
\end{equation}
\end{lemma}

\begin{proof}
Notice that the set  $\Match(\Gamma_{p,q})$ of matchings on $\Gamma_{p,q}$ can be expressed as the disjoint union of  $\Match(\Gamma_{p,q}-\omega)$ and $\{M\cup\{\{v,\omega\}\}\colon M\in\Match(\Gamma_{p,q}-\omega)\}$, where $v$ is a vertex in $X_{p,q}$ which is not matched with any edge in $M$. Moreover, for any given $M\in\Match(\Gamma_{p,q}-\omega)$, there are $p+q-|M|$ choices of $v\in X_{p,q}$. Theses observations yield  \eqref{eq:universal}.
\end{proof}

The second lemma is a formula for $\Psi^{\circ}_{p,q}(x)$.
Let  
\[
Z_{m,n}(x)=\sum_{M\in\Match(K_{m,n})}x^{|M|}
\]
be the matching generating polynomial of the complete bipartite graph $K_{m,n}$. It is easy to verify that 
\[
 Z_{m,n}(x)=\sum_{k\geq0}\binom{m}{k}\binom{n}{k}k!\,x^k.
\] 
Notice that   
\begin{equation}\label{eq:Z-coefficient-form} 
 x^nZ_{m,n}(x^{-1})
 =n!\,[z^n]\e^{xz}(1+z)^m,
\end{equation}
which can be checked as follows:
\begin{align*}
 n!\,[z^n]\e^{xz}(1+z)^m
 &=n!\sum_{k=0}^{n}\frac{x^{n-k}}{(n-k)!}\binom{m}{k}\\
 &=x^n\sum_{k=0}^{n}\binom{m}{k}\binom{n}{k}k!\,x^{-k}
 =x^nZ_{m,n}(x^{-1}).
\end{align*} 

\begin{lemma}\label{lem:complete-bipartite}
For   $p,q\geq1$,
\begin{align}
 \Psi^{\circ}_{p,q}(x)
 =x^{p+q}Z_{p,q-1}(x^{-1})Z_{q,p-1}(x^{-1}).
 \label{eq:Psi-circle-product} 
\end{align}
\end{lemma}

\begin{proof}
This directly  follows from the definition of $ \Psi^{\circ}_{p,q}(x)$  as well as the fact that $ \Gamma_{p,q}-\omega$ is the disjoint union of   $K_{p,q-1}$ and  $K_{q,p-1}$.
\end{proof}

Consider the generating function of $\Psi^{\circ}_{p,q}(x)$:
\[
 \mathcal Q(u,v;x)=
 \sum_{p,q\geq1}\Psi^{\circ}_{p,q}(x)
 \frac{u^{p-1}}{(p-1)!}\frac{v^{q-1}}{(q-1)!}.
\]
Denote 
\begin{equation*} 
 \Phi(u,v)=\frac{u+v+2uv}{1-uv},
 \qquad K(u,v;x)=\exp\bigl(x\Phi(u,v)\bigr).
\end{equation*}

The third lemma formulates $\mathcal Q(u,v;x)$ by using $K(u,v;x)$.

\begin{lemma}\label{lem:shifted-algebraic}
We have 
\begin{equation}\label{eq:Q-algebraic}
 \mathcal Q(u,v;x)
 =K(u,v;x)\frac{x^2(1+u)(1+v)}{(1-uv)^3}.
\end{equation}
\end{lemma}

\begin{proof}
Write $r=p-1$ and $s=q-1$.  Equation
\eqref{eq:Psi-circle-product} becomes
\[
 \Psi^{\circ}_{p,q}(x)
 =x^2\cdot x^sZ_{r+1,s}(x^{-1})\cdot 
       x^rZ_{s+1,r}(x^{-1}).
\]
Applying \eqref{eq:Z-coefficient-form} gives 
\begin{align*}
 x^sZ_{r+1,s}(x^{-1})
   &=s!\,[\eta^s]\e^{x\eta}(1+\eta)^{r+1},\\
 x^rZ_{s+1,r}(x^{-1})
   &=r!\,[\xi^r]\e^{x\xi}(1+\xi)^{s+1},
\end{align*}
and hence 
\begin{equation*}
 \Psi^{\circ}_{p,q}(x)
 =x^2r!s!\,[\xi^r\eta^s]
 \e^{x(\xi+\eta)}(1+\xi)^{s+1}(1+\eta)^{r+1}.
\end{equation*}

Applying Lemma \ref{lem:coefficient-substitution} below to
$F(\xi,\eta)=\e^{x(\xi+\eta)}$, we
obtain
\[
 \frac{\Psi^{\circ}_{p,q}(x)}{r!s!}
 =x^2[u^rv^s]J(u,v)
   \e^{x(\xi(u,v)+\eta(u,v))},
\]
where $\xi(u,v)$ and $\eta(u,v)$  are defined as in \eqref{nj-9-mk9}. 
Hence we see that 
\begin{align}
 \mathcal Q(u,v;x)
 &=\sum_{r,s\geq0}\frac{\Psi^{\circ}_{r+1,s+1}(x)}{r!s!}u^rv^s\notag\\
 &=x^2J(u,v)\e^{x(\xi(u,v)+\eta(u,v))},
 \label{eq:Q-change-variables}
\end{align}
where $J(u,v)$ is the Jacobian determinant  of $\xi(u,v)$ and $\eta(u,v)$ as given in  \eqref{bn-sn-09}. 
Substituting the following relation 
\[
 \xi(u,v)+\eta(u,v)
 =\frac{u(1+v)+v(1+u)}{1-uv}
 =\frac{u+v+2uv}{1-uv}=\Phi(u,v)
\]
 and  \eqref{bn-sn-09} 
into 
\eqref{eq:Q-change-variables}, we are led to  \eqref{eq:Q-algebraic}.
\end{proof}

The following  is a technical lemma that has been used in Lemma \ref{lem:shifted-algebraic}.

\begin{lemma}
\label{lem:coefficient-substitution}
Let $ D=1-uv$, and 
\begin{equation}\label{nj-9-mk9}
    \xi(u,v)=\frac{u(1+v)}{D},\qquad
 \eta(u,v)=v(1+\xi(u,v))=\frac{v(1+u)}{D},
\end{equation}
with Jacobian determinant
\begin{equation}\label{bn-sn-09}
 J(u,v)=\det\frac{\partial(\xi,\eta)}{\partial(u,v)}
       =\frac{(1+u)(1+v)}{D^3}.    
\end{equation}
Then, for every formal power series $F(\xi,\eta)$ and all $r,s\geq0$,
\begin{equation}\label{eq:coefficient-substitution}
 [u^rv^s]\,J(u,v)F(\xi,\eta)
 = [\xi^r\eta^s]
   (1+\xi)^{s+1}(1+\eta)^{r+1}F(\xi,\eta),
\end{equation}
where, on the right-hand side, $\xi$ and $\eta$ are treated as independent
  variables.
\end{lemma}

\begin{proof}
We first need  a one-variable identity.  Suppose that
\[
 z=\frac{ct}{1-dt}, 
\]
where $c$ is invertible.  Let $H(z)=\sum_{n\geq0}h_nz^n$ be a formal power series.  Then, for
$r\geq0$,
\begin{equation}\label{eq:mobius-coefficient}
 [t^r]\frac{\partial z}{\partial t}\cdot H(z)
 =[z^r](c+dz)^{r+1}H(z).
\end{equation}
This can be justified as follows.  Since
\[
 \frac{\partial z}{\partial t}=\frac{c}{(1-dt)^2},
\]
the contribution of each $h_n$ ($0\leq n\leq r$) to the left-hand side of \eqref{eq:mobius-coefficient} is
\[
 h_n[t^{r-n}]\frac{c^{n+1}}{(1-dt)^{n+2}}
 =h_n\binom{r+1}{r-n}c^{n+1}d^{r-n},
\]
which exactly agrees with the  contribution of $h_n$ to the right-hand side  of \eqref{eq:mobius-coefficient}. 
This verifies \eqref{eq:mobius-coefficient}. 

We now apply \eqref{eq:mobius-coefficient} to prove \eqref{eq:coefficient-substitution}.
Set  $c=1+v$, $d=v$, and 
\begin{equation}\label{x-12}
     z=\frac{cu}{1-du}=\xi.
\end{equation}
Note also that 
\[
v(1+z)=\eta.
\] 
Let 
\begin{align}
     H(z)&=(1+z)F\bigl(z,v(1+z)\bigr)\nonumber\\
     &=(1+\xi)F(\xi, \eta).\label{x-09}
\end{align}
Applying 
\eqref{eq:mobius-coefficient} with $t$ replaced by the variable $u$,
\begin{equation}\label{eq:mobius-coe}
 [u^r]\frac{\partial z}{\partial u}\cdot H(z)
 =[z^r](c+dz)^{r+1}H(z).
\end{equation}

By \eqref{x-12} and \eqref{x-09}, one may compute that 
\begin{equation*}
  \frac{\partial z}{\partial u}\cdot H(z)=J(u,v) F(\xi,\eta).  
\end{equation*}
On the other hand, noting that 
  $c+dz=1+\eta$, we have
\begin{equation*}
    (c+dz)^{r+1}H(z)=(1+\eta)^{r+1}(1+\xi)F(\xi,\eta),
\end{equation*}  
{where $\eta=v(1+\xi)$}. 
So \eqref{eq:mobius-coe} can be rewritten as 
\begin{equation}\label{eq:first-coefficient-step}
 [u^r]J(u,v)F(\xi,\eta)
 =[\xi^r](1+\eta)^{r+1}(1+\xi)F(\xi,\eta),
\end{equation}
where, on the right-hand side,  $\eta=v(1+\xi)$ is regarded as a function in~$\xi$.

Finally, for every series $G(\eta)$, note that 
\[
 [v^s]G\bigl(v(1+\xi)\bigr)
 =(1+\xi)^s[\eta^s]G(\eta).
\]
Therefore, taking $[v^s]$ in \eqref{eq:first-coefficient-step}   yields \eqref{eq:coefficient-substitution}.
\end{proof}

\subsection{Proof of Theorem \ref{thm:monomial-model}}

With the above lemmas, we can now give a proof of Theorem \ref{thm:monomial-model}. 

\begin{proof}[Proof of Theorem~\ref{thm:monomial-model}]
Set $A=(1+u)(1+v)$. Recall that 
\[
 \Phi(u,v)=\frac{u+v+2uv}{1-uv}.
\]
So we have $1+\Phi=A/D$.  By
Lemma \ref{lem:universal} and Lemma \ref{lem:shifted-algebraic},
\begin{align}
 \sum_{p,q\geq1}\Psi_{p,q}(x)
 \frac{u^{p-1}}{(p-1)!}\frac{v^{q-1}}{(q-1)!}
 &=(1+\partial_x)\mathcal Q(u,v;x)\notag\\
 &=(1+\partial_x)\left(\e^{\Phi x}\cdot x^2\cdot \frac{A}{D^3}\right)\notag\\
 &=K(u,v;x)\left(\frac{2xA}{D^3}+\frac{x^2A^2}{D^4}\right).
 \label{eq:Psi-shifted-series}
\end{align}

On the other hand, we claim that 
\begin{equation}\label{eq:kernel}
 \sum_{p,q\ge0}(x^p\star x^q)
     \frac{u^p}{p!}\frac{v^q}{q!}
 =K(u,v;x).
\end{equation}
This can be shown as follows.
First, via direct computation (see for example \cite[Proposition 5.17(b) and (f)]{BRW94}), we have 
\begin{equation}\label{eq:Lah-egfs}
 \sum_{n\ge0}\Lah(y^n)\frac{u^n}{n!}
   =\exp\left(\frac{xu}{1-u}\right)
\end{equation}
and 
\begin{equation}\label{eq:Lah-egfs-2}
 \sum_{n\ge0}\Lah^{-1}(x^n)\frac{v^n}{n!}
   =\exp\left(\frac{yv}{1+v}\right).
\end{equation}
Recalling that $x^p\star x^q
 =\Lah\bigl(\Lah^{-1}(x^p)\Lah^{-1}(x^q)\bigr)$
and using  \eqref{eq:Lah-egfs-2},  the left-hand side of
\eqref{eq:kernel} equals
\begin{align*}
 \Lah\left(
  \exp\left(\frac{yu}{1+u}\right)
  \exp\left(\frac{yv}{1+v}\right)
 \right)
 &=\Lah\left(
  \exp\left(y\left(\frac{u}{1+u}+\frac{v}{1+v}\right)\right)
 \right) \\
 &=\Lah\sum_{n\geq 0}y^n\frac{\left(\frac{u}{1+u}+\frac{v}{1+v}\right)^n}{n!}\\
 &=\exp\left(
  x\frac{\frac{u}{1+u}+\frac{v}{1+v}}
          {1-\frac{u}{1+u}-\frac{v}{1+v}}
 \right),
\end{align*}
where the last equality used \eqref{eq:Lah-egfs}. Then \eqref{eq:kernel} is obtained   by noticing  that 
\[
\frac{\frac{u}{1+u}+\frac{v}{1+v}}
          {1-\frac{u}{1+u}-\frac{v}{1+v}}=\frac{u+v+2uv}{1-uv}=\Phi(u,v).
\]

Comparing  \eqref{eq:Psi-shifted-series} with
\eqref{eq:kernel}, to conclude \eqref{eq:monomial-model}, it suffices to verify that 
the second partial derivative $K_{uv}$ of $K(u,v;x)$ is equal to 
\[
 K(u,v;x)\left(\frac{2xA}{D^3}+\frac{x^2A^2}{D^4}\right).
\]
The calculation is routine and thus omitted.  
\end{proof}

\section{Proof of Theorem \ref{thm:abstract-main}}\label{se-4-k}

For the  purpose of proving Theorem \ref{thm:abstract-main}, we consider a bipartite graph obtained from    $\Gamma_{p,q}$ by adding edges 
\[
 \{r_i,\widehat r_i\}\quad(1\leq i<p),
 \qquad
 \{b_j,\widehat b_j\}\quad(1\leq j<q).
\]
Let 
$\boldsymbol\alpha=(\alpha_1,\ldots,\alpha_{p-1})$ and
$\boldsymbol\beta=(\beta_1,\ldots,\beta_{q-1})$ be two sequences of variables.  Endow the edges 
\[
 \{r_i,\widehat r_i\}\quad(1\leq i<p),
 \qquad
 \{b_j,\widehat b_j\}\quad(1\leq j<q)
\]
with weights $\alpha_i$ and $\beta_j$, respectively, and  all original edges in $\Gamma_{p,q}$
with  weight $1$.  Denote this weighted bipartite graph by
$\Gamma_{p,q}(\boldsymbol\alpha,\boldsymbol\beta)$. In  Figure \ref{fig:weighted-Gamma23}, we illustrate  the weighted bipartite graph $\Gamma_{2,3}(\boldsymbol\alpha,\boldsymbol\beta)$ where dashed  edges are the extra added edges from $\Gamma_{2,3}$.

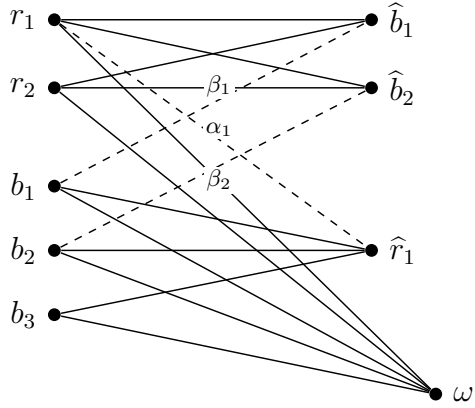
\begin{figure}[htbp]
\centering
\begin{tikzpicture}[
    dot/.style={circle,fill=black,inner sep=1.7pt},
    edge/.style={line width=0.55pt},
    weight/.style={
        fill=white,
        inner sep=1pt,
        font=\scriptsize
    }
]

% The K_{2,2} component.
\node[dot,label=left:$r_1$] (r1) at (0,3.3) {};
\node[dot,label=left:$r_2$] (r2) at (0,2.4) {};
\node[dot,label=right:$\widehat b_1$] (hb1) at (4.2,3.3) {};
\node[dot,label=right:$\widehat b_2$] (hb2) at (4.2,2.4) {};

\foreach \r in {r1,r2}
  \foreach \b in {hb1,hb2}
    \draw[edge] (\r)--(\b);

% The K_{3,1} component.
\node[dot,label=left:$b_1$] (b1) at (0,1.1) {};
\node[dot,label=left:$b_2$] (b2) at (0,0.25) {};
\node[dot,label=left:$b_3$] (b3) at (0,-0.60) {};
\node[dot,label=right:$\widehat r_1$] (hr1) at (4.2,0.25) {};

\foreach \b in {b1,b2,b3}
  \draw[edge] (\b)--(hr1);

% The universal vertex.
\node[dot,label=right:$\omega$] (w) at (5.05,-1.65) {};

\foreach \v in {r1,r2,b1,b2,b3}
  \draw[edge] (\v)--(w);

% The three added weighted edges.
\draw[edge,dashed]
  (r1)--node[weight,pos=0.52,above]{$\alpha_1$}(hr1);

\draw[edge,dashed]
  (b1)--node[weight,pos=0.52,above]{$\beta_1$}(hb1);

\draw[edge,dashed]
  (b2)--node[weight,pos=0.52,below]{$\beta_2$}(hb2);

\end{tikzpicture}

\caption{The weighted bipartite  graph
$\Gamma_{2,3}(\boldsymbol\alpha,\boldsymbol\beta)$.
The dashed edges have weights
$\alpha_1$, $\beta_1$ and $\beta_2$, respectively.}
\label{fig:weighted-Gamma23}
\end{figure}

Let
\[
w(M)=\prod_{e\in M}w_e
\]
denote the weight of  a matching  $M$ of $\Gamma_{p,q}(\boldsymbol\alpha,\boldsymbol\beta)$. Applying    Theorem \ref{thm:monomial-model}, we are able to obtain the following combinatoiral formula. 

\begin{theorem}\label{thm:weighted-model}
For $p,q\geq 1$,
\begin{align}
\left(x\prod_{i=1}^{p-1}(x+\alpha_i)\right)\star
   &\left(x\prod_{j=1}^{q-1}(x+\beta_j)\right)\notag\\[5pt]
 &\qquad=
 \sum_{M\in\Match(\Gamma_{p,q}(\boldsymbol\alpha,\boldsymbol\beta))}
 w(M)x^{p+q-|M|}.
 \label{eq:weighted-model}
\end{align}
\end{theorem}

\begin{proof}
For $S\subseteq \{1,\ldots, p-1\}$ and $T\subseteq \{1,\ldots, q-1\}$, write
$\alpha_S=\prod_{i\in S}\alpha_i$ and
$\beta_T=\prod_{j\in T}\beta_j$. 
By the construction of $\Match(\Gamma_{p,q}(\boldsymbol\alpha,\boldsymbol\beta))$, it is readily checked that 
\[
 [\alpha_S\beta_T]
 \sum_{M\in\Match(\Gamma_{p,q}(\boldsymbol\alpha,\boldsymbol\beta))} w(M)x^{p+q-|M|}
 =\Psi_{p-|S|,q-|T|}(x),
\]
which, in view of Theorem \ref{thm:monomial-model}, gives 
\begin{equation}\label{bn-0-8}
 [\alpha_S\beta_T]
 \sum_{M\in\Match(\Gamma_{p,q}(\boldsymbol\alpha,\boldsymbol\beta))} w(M)x^{p+q-|M|}
 =x^{p-|S|}\star x^{q-|T|}.
\end{equation}

On the other hand, we see that
\begin{align*}
 x\prod_{i=1}^{p-1}(x+\alpha_i)
   &=\sum_{S\subseteq\{1,\ldots, p-1\}}\alpha_Sx^{p-|S|},\\
 x\prod_{j=1}^{q-1}(x+\beta_j)
   &=\sum_{T\subseteq\{1,\ldots, q-1\}}\beta_Tx^{q-|T|}.
\end{align*}
Hence  the coefficient of
$\alpha_S\beta_T$ on the left-hand side of \eqref{eq:weighted-model} is equal to $x^{p-|S|}\star x^{q-|T|},$ which, along with 
\eqref{bn-0-8}, leads to the formula in  \eqref{eq:weighted-model}. 
\end{proof}

We still need a classical  real-rootedness result due to  Heilmann and Lieb \cite{HL72}.
Suppose that $B=(V,E)$ is a   graph with nonnegative edge weights $(w_e)_{e\in E}$.
Let
\begin{equation*}
 \mu_{B,w}(x)=\sum_{M\in\Match(B)}
 (-1)^{|M|}w(M)x^{|V|-2|M|}.
\end{equation*}
The following result appeared  in \cite[Theorem 4.2]{HL72}, see also  \cite[Theorem 2.2 and Remark~2.3]{Amini19}.

\begin{theorem}[{\cite[Theorem 4.2]{HL72}}]
\label{thm:heilmann-lieb}
Let $B=(V,E)$ be a   graph with nonnegative edge weights.
Then $\mu_{B,w}(x)$ has only real zeros.
\end{theorem}

As a direct application  of Theorem 
\ref{thm:heilmann-lieb}, we have the following corollary. 

\begin{corollary}\label{cor:matching-generating}
Let $B=(V,E)$ be a   graph with nonnegative edge weights.
The weighted matching generating polynomial
\begin{equation*}
 Z_{B,w}(x)=\sum_{M\in\Match(B)}w(M)x^{|M|}
\end{equation*}
has only real nonpositive zeros.  Hence, for every
$N\geq\deg Z_{B,w}$, the reciprocal polynomial
$x^NZ_{B,w}(1/x)$ has only real nonpositive zeros.
\end{corollary}

\begin{proof}
By definition, 
\begin{equation}\label{eq:mu-Z}
 \mu_{B,w}(x)=x^{|V|}Z_{B,w}(-x^{-2}).
\end{equation} 
Because 
\[
\mu_{B,w}(-x)=(-1)^{|V|}\mu_{B,w}(x),
\]
it follows from 
Theorem~\ref{thm:heilmann-lieb} that the  nonzero roots of $\mu_{B,w}(x)$ are real and occur in
opposite pairs.  Let $m=\deg Z_{B,w}$.   Since $\mu_{B,w}(x)$ is monic, there exist
$\lambda_1,\ldots,\lambda_m\geq0$ such that
\[
 \mu_{B,w}(x)=x^{|V|-2m}\prod_{i=1}^{m}(x^2-\lambda_i).
\]
Comparing this factorization with \eqref{eq:mu-Z} gives
\[
 Z_{B,w}(x)=\prod_{i=1}^{m}(1+\lambda_i x),
 \]
 and so
\[
 x^NZ_{B,w}(1/x)=x^{N-m}\prod_{i=1}^{m}(x+\lambda_i).
\]
Hence both $ Z_{B,w}(x)$ and $x^NZ_{B,w}(1/x)$ have only real nonpositive zeros. This completes  the proof. 
\end{proof}

We are finally in a position to finish the proof of Theorem \ref{thm:abstract-main}. 

\begin{proof}[Proof of Theorem \ref{thm:abstract-main}]
Without loss of generality, we may  assume that $f$ and $g$
are monic.  Let $p=\deg f$ and $q=\deg g$. Since both $f$ and $g$ have only real nonpositive roots, it follows that 
\[
 f(x)=x\prod_{i=1}^{p-1}(x+\alpha_i),\qquad
 g(x)=x\prod_{j=1}^{q-1}(x+\beta_j),
 \qquad \alpha_i,\beta_j\geq0.
\]
Invoking  Theorem \ref{thm:weighted-model} gives  
\begin{equation*}
 f\star g=x^{p+q}
 Z_{\Gamma_{p,q}(\boldsymbol\alpha,\boldsymbol\beta),w}(1/x).
\end{equation*}
Note also that 
$\deg Z_{\Gamma_{p,q}(\boldsymbol\alpha,\boldsymbol\beta),w}(x)\leq p+q-1$.
Applying 
Corollary \ref{cor:matching-generating} allows concludes  that  all zeros of $f\star g$ belong to  $(-\infty,0]$.
\end{proof}

 \subsection*{Acknowledgements}
The work of Philip B. Zhang was supported by the National Natural Science Foundation of China (Grant No. 12171362) and the Tianjin Municipal Natural Science Foundation (Grant No. 25JCYBJC00430). Sophie C.C. Sun was supported by the Natural Science Foundation of Tianjin (Grant No. 25JCQNJC00250).

\subsection*{Declaration of AI usage}

During the development of this work, we used ChatGPT 5.5 and 5.6 to verify our proposed strategies and to assist with the computational experiments. Several subtle details of the construction and computations were jointly figured out  by ChatGPT and the authors.

\end{document}